\documentclass[11pt]{amsart}
\usepackage{epsfig}
\usepackage{amsmath}
\usepackage{amssymb}
\usepackage{amscd}
\usepackage{color}

\newtheorem{theorem}{Theorem}[section]
\newtheorem{corollary}{Corollary}

\newtheorem{lemma}[theorem]{Lemma}
\newtheorem{proposition}{Proposition}
\theoremstyle{definition}
\newtheorem{definition}[theorem]{Definition}
\newtheorem{remark}{Remark}

\numberwithin{equation}{section}

\begin{document}

\title[partially hyperbolic sets and $ACIP$]{partially hyperbolic sets with positive
measure and $ACIP$ for partially hyperbolic systems}

\author[Pengfei Zhang]{Pengfei Zhang}

\address{Department of Mathematics, University of Science and Technology of China,
 Hefei, Anhui 230026, P. R. China\\
\newline and
CEMA, Central University of Finance and Economics, Beijing 100081,
China }


\email{pfzh311@gmail.com}

\begin{abstract}
In \cite{X} Xia introduced a simple dynamical density basis for
partially hyperbolic sets of volume preserving diffeomorphisms. We
apply the density basis to the study of the topological structure of
partially hyperbolic sets. We show that if $\Lambda$ is a strongly
partially hyperbolic set with positive volume, then $\Lambda$
contains the global stable manifolds over ${\alpha}(\Lambda^d)$ and
the global unstable manifolds over ${\omega}(\Lambda^d)$.

We give several applications of the dynamical density to partially
hyperbolic maps that preserve some $acip$. We show that if $f$ is
essentially accessible and $\mu$ is an $acip$ of $f$, then
$\text{supp}(\mu)=M$, the map $f$ is transitive, and $\mu$-a.e.
$x\in M$ has a dense orbit in $M$. Moreover if $f$ is accessible and
center bunched, then either $f$ preserves a smooth measure or there
is no $acip$ of $f$.
\end{abstract}

\maketitle

\section{Introduction}

Let $M$ be a $n$-dimensional connected, closed manifold, $r>1$ and
$f\in\mathrm{Diff}^r(M)$ be a $C^r$ diffeomorphism on $M$. A compact
$f$-invariant subset $\Lambda\subset M$ is said to be {\it partially
hyperbolic} if there are a nontrivial $Tf$-invariant splitting of
$T_xM = E^s_x\oplus E^{c}_x\oplus E^u_x$ for every $x\in\Lambda$, a
smooth Riemannian metric $g$ on $M$ for which we can choose
continuous positive functions $\nu,\tilde{\nu},\gamma$ and
$\tilde{\gamma}$ on $\Lambda$ with $\nu,\tilde{\nu}<1$ and
$\nu<\gamma\le\tilde{\gamma}^{-1}<\tilde{\nu}^{-1}$ such that, for
all $x\in\Lambda$ and for all unit vectors $v\in E^s_x$, $w\in
E^c_x$ and $v'\in E^u$,
\begin{equation}\label{partial}
\|Tf(v)\|\le\nu(x)<\gamma(x)\le\|Tf(w)\|\le\tilde{\gamma}(x)^{-1}
<\tilde{\nu}^{-1}(x)\le\|Tf(v')\|.
\end{equation}
The notation here is taken from \cite{BW}. Such a metric is called
{\it adapted} (see \cite{G}). If both $E^s$ and $E^u$ are
nontrivial, then we say $\Lambda$ is strongly partially hyperbolic.
In particular the map $f$ is called a (strongly) partially
hyperbolic diffeomorphism if $M$ itself is a (strongly) partially
hyperbolic set. It is well known that $E^s$ and $E^u$ are uniquely
integrable and tangent to the stable lamination $\mathcal{W}^s$ and
the unstable lamination $\mathcal{W}^u$ respectively.

In \cite{X} Xia introduced a simple dynamical density basis for
general partially hyperbolic sets. Namely let $\delta>0$ small,
$W^s(x,\delta)$ be the local stable manifold over $x\in\Lambda$. Let
$B^s_n(p)=f^nW^s(f^{-n}p,\delta)$ for each $p\in\Lambda$ and
$n\ge0$. The collection of sets
$\mathcal{S}=\{B^s_n(p):n\ge0,p\in\Lambda\}$ is called the {\it
stable basis} (see \cite{X}). Let $A\subset\Lambda$ be a measurable
subset. A point $p\in A$ is said to be a $\mathcal{S}$-density point
of $A$ if
\begin{equation*}
\lim_{n\to\infty}\frac{m_{W^s(p)}(B^s_n(p)\cap
A)}{m_{W^s(p)}(B^s_n(p))}=1.
\end{equation*}
where $m_{W^s(p)}$ is the leaf-volume induced by the Riemannian
metric. Let $A^d$ be the set of $\mathcal{S}$-density points of $A$.
Following \cite{X} we have:

\vskip0.3cm

\noindent\textbf{Proposition.} {\it Let $r>1$,
$f\in\mathrm{Diff}^r(M)$ and $\Lambda$ be a partially hyperbolic set
with positive volume. For each measurable subset $A\subset\Lambda$,
$m$-a.e. point in $A$ is $\mathcal{S}$-density point of $A$, that
is, $m(A\backslash A^d)=0$. In words, $\mathcal{S}$ forms a density
basis.}

\vskip0.3cm

This simply defined density basis turns out to be useful in the
study of the topological structure of (partially) hyperbolic sets.
There is an extensive literature discussing the topology of
(partially) hyperbolic sets. We just name a few that are close
related to the results here. Bowen showed in \cite{B}, there exists
$C^1$ horseshoe of positive volume (It is also showed in \cite{B1}
that this {\it fat} horseshoe can not exist among the $C^2$
diffeomorphisms). In \cite{AP} Alves and Pinheiro showed that for a
diffeomorphism $f\in\mathrm{Diff}^r(M)$, if $\Lambda$ is a partially
hyperbolic set that attracts a positive volume set, then $\Lambda$
contains some local unstable disk (hence $\Lambda$ cannot be a
horseshoe-like set). Under a much stronger setting, we can get a
useful characterization that serves well for later applications.
More precisely let $\alpha(x)$ be the set of accumulating points of
$\{f^nx:n\le0\}$ of $x\in M$. For $E\subset M$, let ${\alpha}(E)$ be
the closure of $\bigcup_{x\in E}\alpha(x)$. Similarly we can define
${\omega}(x)$ and ${\omega}(E)$. Then we have

\vskip0.3cm

\noindent\textbf{Theorem A.} {\it Let $f\in\mathrm{Diff}^{r}(M)$ for
some $r>1$ and $\Lambda$ be a partially hyperbolic set with positive
volume. Then $\Lambda$ contains the global stable manifolds over
${\alpha}(\Lambda^d)$, that is, $W^s(x)\subset\Lambda$ for each
$x\in {\alpha}(\Lambda^d)$.

In particular if $\Lambda$ is a strongly partially hyperbolic set
with positive volume, then $\Lambda$ contains the global stable
manifolds over ${\alpha}(\Lambda^d)$ and the global unstable
manifolds over ${\omega}(\Lambda^d)$. }

\vskip0.3cm

The argument here relies on the bounded distortion estimates and the
absolute continuity of stable and unstable laminations, which fail
for $C^1$ maps. See \cite{B,RY}. Although ${\alpha}(\Lambda^d)$ is
nonempty, the volume of ${\alpha}(\Lambda^d)$ could be zero (even in
the hyperbolic case). In fact Fisher \cite{F} constructed several
hyperbolic sets $\Lambda$ with nonempty interior such that
$\alpha(\Lambda^d)$ are repellers and $\omega(\Lambda^d)$ are
attractors (hence their volume must be zero).

A point $x$ is said to be backward recurrent if $x\in\alpha(x)$, and
to be recurrent if $x\in\alpha(x)\cap\omega(x)$. An interesting case
is when most points are recurrent. This will hold in particular if
$\mu(\Lambda)>0$ for some {\it absolutely continuous invariant
probability} measure ({\it acip} for short) $\mu\ll m$. For
simplicity we assume that $\Lambda=\text{supp}\mu$.

\vskip0.3cm

\noindent\textbf{Corollary B.} {\it Let $f\in\mathrm{Diff}^{r}(M)$
for some $r>1$ and $\Lambda$ a strongly partially hyperbolic set
supporting some $acip$ $\mu$. Then $\Lambda$ is bi-saturated, that
is, for each point $p\in\Lambda$, the global stable manifolds and
the global unstable manifolds over $p$ lie in $\Lambda$. }

\vskip0.3cm

In particular we give a dichotomy for maps
$f\in\mathrm{Diff}^{r}(M)$: either $f$ is a transitive Anosov
diffeomorphism, or each $f$-invariant hyperbolic set $\Lambda$ is
$acip$-null, that is, $\mu(\Lambda)=0$ for every $acip$ $\mu$.

\vskip0.3cm

\noindent\textbf{Theorem C.} {\it Let $f\in\mathrm{Diff}^{r}(M)$ for
some $r>1$, $\mu$ be an $acip$ and $\Lambda$ be a hyperbolic set
with positive $\mu$-measure. Then $\Lambda=M$ and $f$ is a
transitive Anosov diffeomorphism on $M$. }

\vskip0.3cm

The similar result has been proved if the $acip$ $\mu$ assumed to be
equivalent to $m$ (see \cite{BocV,X}). Moreover it is proved in
\cite{B1} that for $C^r$ transitive Anosov diffeomorphism, the
$acip$ must have H\"older continuous density with respect to the
volume and be an ergodic (indeed $Bernoulli$) measure, see Remark
\ref{bernoulli}. Also note that the condition that $\Lambda$ has
positive $\mu$-measure for some $acip$ is nontrivial and see
\cite{F} for counter-examples.

Theorem C motivates the analogous generalizations from hyperbolic
systems to accessible partially hyperbolic systems. Recall that an
$f$-invariant measure $\mu$ is said to be {\it weakly ergodic} if
for $\mu$-a.e. $x$, $\mathcal{O}(x)$ is dense in $\text{supp}(\mu)$.
Following generalizes the well known result of Brin \cite{Br} to the
presence of $acip$.

\vskip0.3cm

\noindent\textbf{Theorem D.} {\it Let $f\in\mathrm{SPH}^r(M)$ for
some $r>1$ be essentially accessible. If there exists some $acip$
$\mu$ of $f$, then $\mathrm{supp}(\mu)=M$, the map $f$ is
transitive, and the $acip$ $\mu$ is weakly ergodic. In particular
$\mathcal{O}(x)$ is dense in $M$ for $\mu$-a.e. $x\in M$. }

\vskip0.3cm

In the following we assume $r=2$ for simplicity. Burns and Wilkinson
proved in \cite{BW} that if a map $f\in\mathrm{SPH}^2(M)$ is center
bunched, then every measurable bi-essential saturated set is
essential bi-saturated. Applying to $acip$ we have

\vskip0.3cm

\noindent\textbf{Proposition.} {\it Let $f\in \mathrm{SPH}^2(M)$ be
essentially accessible and center bunched. If there exists some
$acip$ $\mu$, then $\mu$ must be equivalent to the volume. In
particular $\mu$ is ergodic. }

\vskip0.3cm

Note that the arguments in \cite{BW} still work if $f\in
\mathrm{SPH}^{r}(M)$ for $r>1$, as long as we assume {\it strong
center bunching} (see \cite[Theorem 0.3]{BW}). So our results also
extend to this setting. Then applying the {\it cohomologous theory}
developed in \cite{W}, we show that the $acip$ is a {\it smooth
measure}, that is, the density $\frac{d\mu}{dm}$ of $\mu$ with
respect to $m$ is H\"older continuous on $M$, bounded and bounded
away from zero.

\vskip0.3cm

\noindent\textbf{Theorem E.} {\it Let $f\in \mathrm{SPH}^2(M)$ be
accessible and center bunched. If there exists some $acip$, then the
$acip$ must have H\"older continuous density with respect to the
volume of $M$. In words, either $f$ preserves some smooth measure or
there is no $acip$ for $f$. }

\vskip0.3cm

Combining the results in \cite{DW} we have the following direct
corollary:

\vskip0.3cm

\noindent\textbf{Corollary F.} {\it The set of maps that admit no
$acip$ contains a $C^1$ open and dense subset of $C^2$ strongly
partially hyperbolic and center bunched diffeomorphisms. In
particular the set of maps that admit no $acip$ contains a $C^1$
open and dense subset of $C^2$ strongly partially hyperbolic
diffeomorphisms with $\dim(E^c)=1$. }

\vskip0.3cm

Finally we remark that although the volume measure need not be
$f$-invariant, there always exists some $f$-invariant measures. The
density argument combines the dynamics of $acip$ and the dynamics of
volume on $M$. This is why most results of volume-preserving
partially hyperbolic systems have parallel generalizations to the
systems with $acip$.

\section{Dynamical density basis for partially hyperbolic set}

In this section we will consider a $C^r$ diffeomorphism for some
$r>1$ and a partially hyperbolic invariant set with positive volume.
More precisely let $M$ be a closed and connected smooth manifold.
Each Riemannian metric $g$ on $M$ induces a (geodesic) distance $d$
on $M$ and a normalized volume measure $m$ on $M$. Let $\mathcal{B}$
be the Borel $\sigma$-algebra of $M$. A Borel probability measure
$\mu$ on $M$ is said to be {\it absolutely continuous} with respect
to $m$, denoting $\mu\ll m$, if $\mu(A)=0$ for each set $A\in
\mathcal{B}$ with $m(A)=0$, and to be {\it equivalent to} $m$ if
$\mu\ll m$ and $m\ll \mu$. It is evident for any other Riemannian
metric $g'$ compatible with $g$, the induced volume of $g'$ is
equivalent to $m$.

Let $f\in\mathrm{Diff}^r(M)$ for $r>1$ and $\Lambda$ be a compact
partially hyperbolic invariant set with positive volume. In the
following we always assume that the stable subbundle $E^s$ is
nontrivial on $\Lambda$ and $m$ is the normalized volume measure on
$M$ induced by some Riemannian metric adapted to the invariant
splitting (see \cite{G}).

Since $r>1$, it is well known that the stable bundle $E^s$ is
H\"older continuous over $\Lambda$ (the H\"older exponent may be
much smaller than $r-1$, see \cite{BS}) and is tangent to the stable
lamination $\mathcal{W}^s$ over $\Lambda$. (A lamination over
$\Lambda$ is a partial foliation which may not foliate an open
neighborhood of $\Lambda$, see \cite{HPS}. In case that $\Lambda=M$,
$\mathcal{W}^s$ turns out to be a foliation.) For $\delta>0$ small
we use $W^s(x,\delta)$ to denote the local stable manifold over
$x\in\Lambda$. Note that $W^s(x,\delta)$ varies uniformly
$\alpha$-H\"older continuously with respect to the base point
$x\in\Lambda$. For simplicity we also write $E_y^s=T_yW^s(x)$ for
all $y\in W^s(x,\delta)$ and $x\in\Lambda$. The invariance of
$\mathcal{W}^s$ implies that the extended distribution $E^s$ is also
invariant. The H\"older continuity of $E^s$ ensures that the family
$\mathcal{W}^s$ is absolutely continuous. By slightly increasing
$\nu$ and decreasing $\delta$ if necessary, we can assume that for
each $x\in\Lambda$ the following holds:
\begin{equation}\label{stable}
\text{if }p,p'\in W^s(x,\delta),\text{ then
}d(fp,fp')\le\nu(p)d(p,p').
\end{equation}
In particular we have $fW^s(x,\delta)\subseteq
W^s(fx,\delta\!\cdot\!\nu(x))$ for all $x\in \Lambda$.

Before moving on, let's fix some notations as in \cite{BW}. Let
$S\subset M$ be a submanifold of $M$, $m_S$ be the volume measure on
$S$ induced by the restricted Riemannian metric $g|_{S}$ on $S$. In
particular if $S=W^s(x)$, we abbreviate the induced measure as
$m_{s,x}$. Denote $m_{s,x}(A)$ the restricted submanifold measure
for a measurable subset $A\subseteq W^s(x)$. This should not be
confused with conditional measures. Let $\eta=\min\{\|Tf(v)\|:v\in
TM\text{ with }\|v\|=1\}$ and
$\overline{\nu}=\sup_{p\in\Lambda}\nu(p)$. Clearly
$0<\eta\le\nu(p)\le\overline{\nu}<1$ by compactness. For each $p\in
\Lambda$ we let $p_i=f^ip$ for $i\in\mathbb{Z}$, $\nu_0(p)=1$ and
$\nu_n(p)=\nu(p_{n-1})\cdots\nu(p_0)$ for $n\ge1$. Let
$B^s_n(p)=f^nW^s(p_{-n},\delta)$. Since $\Lambda$ is $f$-invariant,
we have $B^s_n(p)\subset W^s(p,\delta\!\cdot\! \nu_n(p_{-n}))$.

Since each stable manifold is a smooth submanifold of the Riemannian
manifold $M$ and $f$ is $C^r$ for $r>1$, the {\it stable Jacobian}
$J^s(f,x)$ of the restricted map $Tf:E^s_x\rightarrow E^s_{x_1}$ is
well defined and H\"older continuous with uniform H\"older exponent
and H\"older constant. That is, there exist $\alpha>0$ and $C_0>0$
such that for any $p\in\Lambda$ and $x,y\in W^s(p,\delta)$ we have
$|J^s(f,x)-J^s(f,y)|\le C_0 d(x,y)^\alpha$. Also there exists
$J^*\ge1$ such that $1/J^*\le J^s(f,x)\le J^*$ for all $x\in
W^s(p,\delta)$ and $p\in\Lambda$. Decreasing $\delta$ again if
necessary we assume $C_1=\prod_{k=0}^{\infty}\frac{(1+J^*C_0
\delta^\alpha \overline{\nu}^{k\alpha})}{(1-J^*C_0 \delta^\alpha
\overline{\nu}^{k\alpha})}<\infty$.

Let $B^s_n(p)=f^nW^s(f^{-n}p,\delta)$ for each $p\in\Lambda$,
$n\ge0$ and $\mathcal{S}=\{B^s_n(p):n\ge0,p\in\Lambda\}$ be the
stable basis. It is easy to see that $\{B^s_n(p):n\ge0\}$ forms a
nesting sequence of neighborhoods of $p\in\Lambda$ relative to
$W^s(p,\delta)$ and $B^s_n(p)$ shrinks to $p$ as
$n\rightarrow\infty$. Note that the basis here is in leafwise sense
and may has infinite eccentricity. The proposition below states that
the stable basis $\mathcal{S}$ behaves well in the sense of
\cite{PS}:
\begin{proposition}\label{basis}
The following properties hold for stable basis $\mathcal{S}$:
\begin{enumerate}
\item For any $p\in\Lambda$, $m_{s,p}(B^s_n(p))\rightarrow0$
if and only if $n\rightarrow\infty$.

 \item For any $k\ge0$, there exists $c_k\ge1$ such that
$\frac{m_{s,p}(B^s_n(p))} {m_{s,p}(B^s_{n+k}(p))}\le c_k$ for all
$p\in\Lambda$, $n\ge0$.

\item There exists $L\in\mathbb{N}$ such that for any $p,q\in\Lambda$,
$n\ge0$, if $B^s_{n+L}(p)\cap B^s_{n+L}(q)\neq\emptyset$, then
$B^s_{n+L}(p)\cup B^s_{n+L}(q)\subseteq B^s_{n}(p)\cap B^s_{n}(q)$.
\end{enumerate}
\end{proposition}
The properties listed above appeared in \cite{PS} (in a general
setting) and is named to be {\it volumetrically engulfing} (also see
\cite{X} for example). The proof mainly uses distortion estimates.

Let $A\in\mathcal{B}_{\Lambda}$ be a measurable subset of $\Lambda$.
Recall that a point $x\in A$ is said to be a $\mathcal{S}$-density
point of $A$ if
\begin{equation}\label{density}
\lim_{n\to\infty}\frac{m_{s,p}(B^s_n(p)\cap
A)}{m_{s,p}(B^s_n(p))}=1.
\end{equation}
For different $\delta$'s, the induced stable bases are {\it
internested} (see \cite[Lemma 2.1]{BW} for details). So the
definition of $\mathcal{S}$-density point is independent of the
choice of the size of stable manifolds and the choice of adapted
Riemannian metric on $M$. Let $A^d$ be the set of
$\mathcal{S}$-density points of $A$.

For each $A\in\mathcal{B}_{\Lambda}$ and each $p\in\Lambda$, $A\cap
W^s(p,\delta)$, the intersection of two Borel measurable subsets, is
a Borel measurable subset of the submanifold $W^s(p,\delta)$. (Note
that if $A$ is Lebesgue measurable, above relation will hold for
$m$-a.e. $p\in\Lambda$ by Fubini Theorem.) Let us denote $A^d_p$ the
set of $\mathcal{S}$-density points of $A\cap W^s(p,\delta)$.
Clearly we have $A^d=\bigcup_{p\in\Lambda}A^d_p$.
\begin{proposition}\label{dup}
Let $f\in\mathrm{Diff}^{r}(M)$ for some $r>1$ and $\Lambda$ be a
partially hyperbolic set with positive measure. For each subset
$A\in\mathcal{B}_{\Lambda}$, we have
\begin{enumerate}
\item for each $p\in\Lambda$, $m_{s,p}$-a.e. point in
$W^s(p,\delta)\cap A$ is a $\mathcal{S}$-density point of $A$:
$m_{s,p}(W^s(p,\delta)\cap A\backslash A^d_p)=0$.

\item $m$-a.e. point of $A$ is a $\mathcal{S}$-density
point of $A$: $m(A\backslash A^d)=0$.
\end{enumerate}
Moreover if $A\in\mathcal{B}_{\Lambda}$ is $f$-invariant, so is
$A^d$.
\end{proposition}
\begin{proof}
The first item follows by applying Theorem 3.1 in \cite{PS} to
stable basis $\mathcal{S}$ to each intersection $A\cap
W^s(p,\delta)$. Proposition \ref{basis} ensures that $\mathcal{S}$
forms a density basis in this leafwise sense.

Using the absolute continuity of the stable foliation
$\mathcal{W}^s$ and the relation $A^d=\bigcup_{p\in\Lambda}A^d_p$,
we have $m(A\backslash A^d)=0$. Hence $\mathcal{S}$ also forms a
density basis in the ambient sense.

For the last item, we note that each local leaf $W^s(s,\delta)$ is a
$C^r$ submanifold of $M$ and the restriction of $f$ between local
stable manifolds is diffeomorphic onto its image. So $p\in\Lambda$
is a $\mathcal{S}$-density point of $A\cap W^s(p,\delta)$ (or
equally, of $A$) if and only if $fp$ is a $\mathcal{S}$-density
point of $A\cap W^s(fp,\delta)$.
\end{proof}

\section{Topological structure of partially hyperbolic sets}

In this section we give some descriptions of the topological
structure of partially hyperbolic sets with positive volume. As in
Section $2$ we let $M$ be a connected closed manifold,
$f\in\mathrm{Diff}^r(M)$ for some $r>1$ and $\Lambda$ a partially
hyperbolic set with positive volume.

Given a Borel subset $A\subset \Lambda$, we consider a family of
functions $\eta_n$ on $\Lambda$ as
$$\eta_n(x)=m_{s,x}(B^s_n(x)\backslash A)/m_{s,x}(B^s_n(x)).$$
The following result shows the increasing occupation of an invariant
set $A$ in the local stable manifolds along the backward iterates of
an $\mathcal{S}$-density point of $A$.
\begin{lemma}\label{distortion}
There exists a constant $C\ge1$ such that given an $f$-invariant
subset $A\in\mathcal{B}_{\Lambda}$,
$m_{s,x_{-n}}(W^s(x_{-n},\delta)\backslash A) \le C_2\cdot\eta_n(x)$
for each $x\in\Lambda$ and $n\ge0$.
\end{lemma}
\begin{proof}
We only need to adapt the notations in \cite[Lemma 3.2]{X}, since
the proof is essentially the same. Let $A$ be an invariant set and
$x\in\Lambda$ be fixed. Let $B^k_n=f^kW^s(x_{-n},\delta)$ and
$D_n^k=B^k_n\backslash A$ for each $0\le k\le n$. Note that
$B^0_n=W^s(x_{-n},\delta)$ is a local stable leaf and
$B^n_n=B^s_n(x)$ is an element in the stable basis $\mathcal{S}$.
Then using the constant $C_5$ given by \cite[Page 816]{X}, we have
$m_{s,x_{-n}}(D_n^{0}) \le C_5\cdot\eta_n(x)\cdot
m_{s,x_{-n}}(B_n^{0})$. Applying $B_n^{0}=W^s(x_{-n},\delta)$ and
$D_n^{0}=W^s(x_{-n},\delta)\backslash A$, we finish the proof with a
uniform constant $C_2=C_5\cdot\max_{p\in
\Lambda}m_{s,p}(W^s(p,\delta))$.
\end{proof}

Recall that $\alpha(x)$, the $\alpha$-set of $x$, is the set of
accumulating points along the backward orbit $\{x,f^{-1}x,\cdots\}$.
Let ${\alpha}(E)$ be the closure of $\bigcup_{x\in E}\alpha(x)$.
Note that for each point $x\in\Lambda$ and each subset $E\subset
\Lambda$, the sets $\alpha(x)$ and ${\alpha}(E)$ are compact
$f$-invariant subsets of $\Lambda$.
\begin{theorem}
Let $f\in\mathrm{Diff}^{r}(M)$ for some $r>1$ and $\Lambda$ a
partially hyperbolic set with positive volume. Then $\Lambda$
contains the global stable manifolds over ${\alpha}(\Lambda^d)$.
\end{theorem}
\begin{proof}
First let us consider $y\in\alpha(x)$ for some $x\in\Lambda^d$. Pick
a sequence of times $n_i\to+\infty$ such that $x_{-n_i}\to y$
(clearly all these points are in $\Lambda$). By Lemma
\ref{distortion} we have $m_{s,x_{-n}}(W^s(x_{-n},\delta)\backslash
\Lambda)\le C_2\eta_n(x)$. (Note that $\eta_n(x)\to 0$ as
$n\to\infty$.) Passing to a subsequence if necessary, we can assume
that $W^s(x_{-n_i},\delta)\cap\Lambda$ contains a
$\frac{1}{i}$-dense subset $E_{x_{-n_i},i}$ of
$W^s(x_{-n_i},\delta)$. Let
$E=\limsup_{i\to\infty}E_{x_{-n_i},i}=\bigcap_{k\ge1}\overline{\bigcup_{i\ge
k}E_{x_{-n_i},i}}$. It is clear that $E\subset\Lambda$ since
$\Lambda$ is compact. By continuity of the stable manifolds, $E$
contains a dense subset of $W^s(y,\delta)$, and hence
$W^s(y,\delta)\subset E$. So $W^s(y,\delta)\subset \Lambda$ for each
$y\in\alpha(x)$ and each $x\in\Lambda^d$.

Still by the compactness of $\Lambda$, $W^s(y,\delta)\subset\Lambda$
for each $y\in {\alpha}(\Lambda^d)$. By the invariance of $\Lambda$
and ${\alpha}(\Lambda^d)$, the global stable manifolds
$W^s(y)\subset\Lambda$ for each $y\in {\alpha}(\Lambda^d)$.
\end{proof}

Similarly we consider the $\omega$-set and define
${\omega}(\Lambda^d)$. For a strongly partially hyperbolic set we
have
\begin{theorem}\label{main}
Let $f\in\mathrm{Diff}^{r}(M)$ for some $r>1$ and $\Lambda$ a
strongly partially hyperbolic set with positive volume. Then
$\Lambda$ contains the global stable manifolds over
${\alpha}(\Lambda^d)$ and the global unstable manifolds over
${\omega}(\Lambda^d)$.
\end{theorem}

So every partially hyperbolic set with positive volume is far from
being a topological horseshoe-like set. Although the sets
${\alpha}(\Lambda^d)$ and ${\omega}(\Lambda^d)$ are always nonempty,
we do not know how large they could be and when they could intersect
with each other. This can be improved if we require that $\Lambda$
admits some recurrence.

\begin{definition}
A point $x$ is said to be {\it backward recurrent} if
$x\in\alpha(x)$. The definition of {\it forward recurrent} is
analogous. A point is said to be {\it recurrent} if it is both
backward and forward recurrent.
\end{definition}
\begin{definition}
Let $E$ be a measurable subset of $\Lambda$. Then $E$ is said to be
{\it $s$-saturated} if for each $x\in E$, $W^s(x)\subset E$.
Similarly we can define {\it $u$-saturated} sets. Then the set $E$
is {\it bi-saturated} if it is $s$-saturated and $u$-saturated.
\end{definition}

\begin{corollary}\label{sph}
Let $f\in\mathrm{Diff}^{r}(M)$ for some $r>1$ and $\Lambda$ be a
strongly partially hyperbolic set supporting some $acip$ $\mu$. Then
$\Lambda$ is bi-saturated.
\end{corollary}
\begin{proof}
By Poincar\'e recurrence theorem, we have that $\mu$-a.e.
$x\in\Lambda$ is recurrent, that is, $\mu(\mathrm{Rec}_{\Lambda})=1$
where $\mathrm{Rec}_{\Lambda}$ is the set of recurrent points in
$\Lambda$. Also we have $\mu(\Lambda\backslash \Lambda^d)=0$ since
$\mu\ll m$ and $m(\Lambda\backslash \Lambda^d)=0$. So
$\mu(\Lambda^d\cap\mathrm{Rec}_{\Lambda})=1$ and the closed set
${\alpha}(\Lambda^d)$ contains
$\Lambda^d\cap\mathrm{Rec}_{\Lambda}$, which is a subset of full
$\mu$-measure in $\Lambda$ and hence dense in
$\text{supp}\mu=\Lambda$. So ${\alpha}(\Lambda^d)=\Lambda$ and the
set $\Lambda$ is $s$-saturated by Theorem \ref{main}. Similarly we
can show $\Lambda$ is $u$-saturated. This completes the proof.
\end{proof}

\section{Regularity of $acip$: hyperbolic case.}\label{acip1}

In this section we consider the hyperbolic sets. We show that if a
hyperbolic set has positive $acip$-measure, then the map is a
transitive Anosov diffeomorphism. Then it is well known that the
$acip$ is not only equivalent to the volume, but also has smooth
density with respect to the volume. This motivates the
generalization to partial hyperbolic systems in next section.
\begin{theorem}\label{anosov}
Let $f\in\mathrm{Diff}^{r}(M)$ for some $r>1$, $\mu$ be an $acip$
and $\Lambda$ be a hyperbolic set with positive $\mu$-measure. Then
$\Lambda=M$ and $f$ is an transitive Anosov diffeomorphism on $M$.
\end{theorem}
\begin{proof}
By considering $\Lambda_\mu=\Lambda\cap\text{supp}(\mu)$ and
$\mu|_{\Lambda_\mu}$ if necessary, we can assume that
$\Lambda=\text{supp}(\mu)$. By Corollary \ref{sph}, we have that
$\Lambda$ is bi-saturated. By the uniform hyperbolicity of
$\Lambda$, there exists $\epsilon>0$ such that
$B(x,\epsilon)\subset\bigcup_{y\in
W^u(x,\delta)}W^s(y,\delta)\subset\Lambda$ for each $x\in\Lambda$.
So the set $\Lambda$ is both close and open, hence coincides with
the whole manifold $M$. Since the recurrent set is a dense subset of
$\text{supp}(\mu)=\Lambda=M$ and is contained in the nonwandering
set $\Omega(f)$, we have that $\Omega(f)=M$ and $f$ is a transitive
Anosov diffeomorphism on $M$ (by spectrum decomposition theorem, see
\cite{B1}).
\end{proof}
\begin{remark}\label{bernoulli}
Spectrum decomposition theorem actually implies that $f$ is mixing.
Moreover by Corollary 4.13 and Theorem 4.14 in \cite{B1}, $\mu$
coincides with the equilibrium state $\mu_{\phi^u}$ of the potential
$\phi^u(x)=-\log(J^u(f,x))$, and has H\"older continuous density
with respect to $m$. Furthermore the smooth measure $\mu$ is ergodic
and $Bernoulli$.
\end{remark}
\begin{remark}
The regularity of $f\in\mathrm{Diff}^{r}(M)$ for some $r>1$ is an
essential assumption in a two-fold sense. In \cite{RY} Robinson and
Young constructed a $C^1$ Anosov diffeomorphism with non-absolutely
continuous stable and unstable foliations, which does have some
closed invariant set with positive volume. In \cite{B} Bowen
constructed a $C^1$ horseshoe with positive volume and absolutely
continuous local stable and unstable laminations, where the bounded
distortion property in Lemma \ref{distortion} fails.
\end{remark}

\section{Regularity of $acip$: partially hyperbolic
case}\label{acip2}

In this section we show analogous results in Section \ref{acip1}
hold for accessible strongly partially hyperbolic systems. Namely,
let $f\in\mathrm{SPH}^r(M)$ for some $r>1$ be a $C^r$ strongly
partially hyperbolic diffeomorphism and $m$ be the volume measure
associated to some Riemannian metric adapted to the partially
hyperbolic splitting. Let $\mathcal{W}^s$ be the stable foliation
tangent to the stable bundle and $\mathcal{W}^u$ the unstable
foliation tangent to the unstable bundle.

\begin{definition}
Let $E$ be a measurable subset of $M$. Then $E$ is said to be {\it
essentially $s$-saturated} if there exists an $s$-saturated set
$\widehat{E}^s$ with $m(E\triangle \widehat{E}^s)=0$. Similarly we
can define {\it essentially $u$-saturated} sets. The set $E$ is {\it
essentially bi-saturated} if there exists a bi-saturated set
$\widehat{E}^{su}$ with $m(E\triangle \widehat{E}^{su})=0$, and {\it
bi-essentially saturated} if $E$ is essentially $s$-saturated and
essentially $u$-saturated.
\end{definition}

\begin{definition}
A strongly partially hyperbolic diffeomorphism $f: M\to M$ is said
to be {\it accessible} if each nonempty bi-saturated set is the
whole manifold $M$. The map $f$ is {\it essentially accessible} if
every measurable bi-saturated set has either full or zero volume.
\end{definition}

\begin{theorem}\label{tran}
Let $f\in\mathrm{SPH}^r(M)$ be essentially accessible. If there
exists some $acip$ for $f$, then the support of the $acip$ is the
whole manifold and the map $f$ is transitive.
\end{theorem}
Before the proof, we mention that there exists a $C^1$ open set of
accessible but non-transitive diffeomorphisms (see \cite{NT}).
\begin{proof}
Let $\mu$ be an $acip$ of $f$. Then the support $\mathrm{supp}(\mu)$
of $\mu$ is a strongly partially hyperbolic set supporting $\mu$,
hence is a bi-saturated set by Corollary \ref{sph}. Essential
accessibility of $f$ implies that $m(\mathrm{supp}(\mu))=1$. Hence
$\mathrm{supp}(\mu)=M$ since $\mathrm{supp}(\mu)$ is closed.

Suppose on the contrary that $f$ is not transitive. That is, there
exists an $f$-invariant nonempty open set $U$ such that $M\backslash
\overline{U}\neq\emptyset$. So the set $\Lambda=M\backslash U$ is
$f$-invariant, closed with nonempty interior. Hence $\mu(\Lambda)>0$
and $\mu|_{\Lambda}$ is again an $acip$. Corollary \ref{sph} implies
that $\Lambda$ is bi-saturated. Since $f$ is essentially accessible,
we have $m(\Lambda)=1$ and $m(U)=0$. This contradicts the openness
of $U$.
\end{proof}

Generally for a transitive map $f$, the set $\mathrm{Tran}_f$ of
points with dense orbit is measure-theoretic meagre (although
topological residual). In \cite[Section 5.7]{RHRHU} they extracted
following property which can be viewed as a stronger form of
transitivity (or a weak form of ergodicity).
\begin{definition}An $f$-invariant measure $\mu$ is said to be {\it weakly
ergodic} if the set of points with dense orbit in $\text{supp}(\mu)$
has full $\mu$-measure.
\end{definition}
Clearly that ergodicity implies weak ergodicity, and weak ergodicity
implies the transitivity of the subsystem $(f,\text{supp}(\mu))$. In
the following we show some analogous results in \cite{Br,BDP,RHRHU}
hold for $acip$. To this end let us introduce some necessary
notations. Let $\mu$ be an $acip$ of $f\in\mathrm{SPH}^r(M)$ for
some $r>1$ and $\phi=\frac{d\mu}{dm}$ be the {\it Radon-Nikodym
density} of $\mu$ relative to $m$. Note that the {\it Jacobian}
$J_f:M\to\mathbb{R},x\mapsto \mathrm{Jac}(Df:T_xM\to T_{fx}M)$ is a
H\"older continuous function, bounded and bounded away from $0$ on
$M$. Now for each measurable subset $A\subset M$ we have:
$$\int_{A}\phi(x)dm(x)=\mu(A)
=\mu(fA)=\int_{fA}\phi(y) dm(y)=\int_{A}\phi(fx) J_f(x) dm(x).$$ So
the following holds:
\begin{equation}\label{coho1}
\phi(fx) J_f(x)=\phi(x)\text{ for }m-\text{a.e. }x\in M.
\end{equation}
Let us consider the set $E=\{x\in M:\phi(x)>0\}$. Clearly $E$ is
measurable and $m(E)>0$. By \eqref{coho1} we see that $E$ is also
$f$-invariant. Restricted to the set $E$, the measure $m|_{E}$ is
equivalent to $\mu$. So `(P) for $m$-a.e. $x\in E$' is the same as
`(P) for $\mu$-a.e. $x\in E$'. In this case we will say `(P) a.e.
$x\in E$' for short.

\begin{proposition}\label{biess}
Let $f\in\mathrm{SPH}^r(M)$, $\mu$ be an $acip$ with density $\phi$
and $E=\{x\in M:\phi(x)>0\}$. Then $E$ is bi-essentially saturated.
\end{proposition}
Note that all essential saturations are defined with respect the
volume. If $f$ is volume preserving, then every invariant set is
always bi-essentially saturated by Hopf argument. See \cite[Lemma
6.3.2]{BS} and \cite[Theorem 5.5]{RHRHU}.
\begin{proof}
It suffices to prove that $E$ is essentially $s$-saturated. Let
$B^s_n(x)=f^{n}W^s(x_{-n},\delta)$ and $E^d$ be the set of
$\mathcal{S}$-density points of $E$. By Proposition \ref{dup} we
have $m(E\backslash E^d)=0$.

Consider the functions $\eta_n(x)=m_{W^s(x)}(B^s_n(x)\backslash
E)/m_{W^s(x)}(B^s_n(x))$ for $n\ge1$. So $\eta_n(x)\to 0$ as
$n\to+\infty$ for a.e. $x\in E$. For each $\epsilon>0$, there exists
a subset $E_\epsilon\subset E$ with $m(E\backslash
E_\epsilon)<\epsilon$ on which $\eta_n$ converges uniformly to zero.
For a recurrent point $x\in E_\epsilon$, let $n_i$ be the forward
recurrent times of $x$ with respect to $E_\epsilon$, that is,
$f^{n_i}x\in E_\epsilon$. Note that a.e. $x\in E_\epsilon$ is
recurrent.

By Lemma \ref{distortion}, there exists a uniform constant $C_2$
such that for the point $y=f^{n_i}x$ and $n=n_i$ the following
holds:
\begin{equation*}
m_{W^s(x)}(W^s(x,\delta)\backslash E) \le C_2\cdot
\eta_{n_i}(f^{n_i}x).
\end{equation*}
Passing $n_i$ to $\infty$ we have
$m_{W^s(x)}(W^s(x,\delta)\backslash E)=0$ for a.e. $x\in
E_\epsilon$.

Since $\epsilon$ can be arbitrary small, we have
$m_{W^s(x)}(W^s(x,\delta)\backslash E)=0$ for a.e. $x\in E$. Since
$E$ is $f$-invariant and $f$ is smooth between leaves of
$\mathcal{W}^s$, $m_{W^s(x)}(f^{-n}W^s(f^{n}x,\delta)\backslash
E)=0$ for each $n\ge1$ and a.e. $x\in E$. Hence
$m_{W^s(x)}(W^s(x)\backslash E)=0$ for a.e. $x\in E$. It follows
from the absolute continuity of $\mathcal{W}^s$ that $E$ is
essentially $s$-saturated. Similarly we can show $E$ is essentially
$u$-saturated. This completes the proof.
\end{proof}

\begin{theorem}
Let $f\in\mathrm{SPH}^r(M)$ be essentially accessible. Then every
$acip$ is weakly ergodic. In particular if $\mu$ is an $acip$, then
the orbit $\mathcal{O}(x)$ is dense in $M$ for $\mu$-a.e. $x\in M$.
\end{theorem}
This result is well known if the system is volume preserving (see
\cite{Br,BDP,RHRHU}). The idea of the proof is similar to Lemma 5 in
\cite{BDP}. Also see Proposition 5.17 in \cite{RHRHU}.
\begin{proof}
Let $\phi$ be the density of $\mu$ with respect to $m$ and $E=\{x\in
M:\phi(x)>0\}$. By Proposition \ref{biess}, we have $E$ is
bi-essentially saturated. Hence $\overline{E}=\mathrm{supp}(\mu)=M$
by Theorem \ref{tran} since $f$ is essentially accessible.

{\it Step 1.} We will show that for each open ball $B$,
$\mathcal{O}(x)\cap B\neq\emptyset$ for $m$-a.e. point $x\in E$. To
this end we first consider $G(B)$, the subset of points $x$ which
has a neighborhood $U$ of $x$ such that $\mathcal{O}(y)\cap
B\neq\emptyset$ for $m$-a.e. $y\in U\cap E$. Evidently $G(B)$ is a
nonempty open subset (and $f$-invariant).

\noindent{\bf Claim.} $G(B)$ is bi-saturated. So $m(G(B))=1$ since
$f$ is essentially accessible.

\noindent{\it Proof of Claim}. Let us prove $G(B)$ is $s$-saturated.
It suffices to show that $q\in G(B)$ for each $q\in W^s(z,\delta)$
and each $p\in G(B)$, where the size $\delta$ is fixed. So the
justification lies in a local foliation box $X$ of $\mathcal{W}^s$
around $p$. Note that we can replace $E$ by its saturate
$\widehat{E}^s$ in the definition of $G(B)$ since $E$ is essentially
$s$-saturated. For a point $x\in X$, denote $W^s_X(x)$ the component
of $W^s(x)\cap X$ that contains $x$. Let $U$ be a small neighborhood
of $p$ with $\mathcal{O}(y)\cap B\neq\emptyset$ for $m$-a.e. $y\in
U\cap \widehat{E}^s$. Let $R$ be the set of recurrent points $z\in
U\cap \widehat{E}^s$ whose orbits enter $B$. Note that $m(U\cap
\widehat{E}^s\backslash R)=0$ since $m|_{E}$ is equivalent to the
invariant measure $\mu$ and $m(E\triangle\widehat{E}^s)=0$. So we
can pick a smooth transverse $T$ of $\mathcal{W}^s_{X}$ in $U$ such
that $T\cap W^s_U(p)\neq\emptyset$ and $m_T(\widehat{E}^s\backslash
R)=0$, where $m_T$ is the induced volume on $T$ (It is helpful to
keep in mind that $\widehat{E}^s$ is not only essentially
$s$-saturated, but $s$-saturated). Now we have

\begin{enumerate}
\item[(I)] For each $y\in W^s_X(x)$ and $x\in R$, we have $\mathcal{O}(y)\cap
B\neq\emptyset$. This follows from that $d(f^nx,f^ny)\to0$ and the
recurrence of $x$: the orbit of $x$ will enters $B$ infinite many
times.

\item[(I\!I)] The set $\bigcup_{x\in T\cap R}W^s_X(x)$ has full $m$-measure in the
set $\bigcup_{x\in T\cap \widehat{E}^s}W^s_X(x)$. This follows from
that both sets are measurable and $\mathcal{W}^s_X$-saturated,
$\mathcal{W}^s_X$ is an absolutely continuous lamination of $X$ and
$m_T(\widehat{E}^s\backslash R)=0$.

\item[(I\!I\!I)] The set $\bigcup_{x\in T}W^s_X(x)$
contains an open neighborhood $V$ of $q$. This follows from that the
holonomy maps along $\mathcal{W}^s_X$ are homeomorphisms.
\end{enumerate}
Also note that $\bigcup_{x\in T\cap
\widehat{E}^s}W^s_X(x)=\left(\bigcup_{x\in T}W^s_X(x)\right)\cap
\widehat{E}^s$. So $\mathcal{O}(y)\cap B\neq\emptyset$ for $m$-a.e.
$y\in V\cap \widehat{E}^s$. This implies $q\in G(B)$ and hence
$G(B)$ is $s$-saturated. Similarly we have $G(B)$ is also
$u$-saturated and hence $m(G(B))=1$ by the essential accessibility
of $f$. This completes the proof of Claim.

Now let $F(B)=\{x\in E: \mathcal{O}(x)\cap B\neq\emptyset\}$. We
need to show that $m(E\backslash F(B))=0$. To derive a contradiction
we assume $m(E\backslash F(B))>0$ and $p\in G(B)$ be a Lebesgue
density point of $E\backslash F(B)$ (here we use $m(G(B))=1$). So
there exists an open neighborhood $U$ of $p$ such that
$\mathcal{O}(x)\cap B\neq\emptyset$ for a.e. $x\in U\cap E$. Then we
have $m(U\cap E\backslash F(B))=0$. But this is impossible since we
choose $p$ as a Lebesgue density point of $E\backslash F(B)$. So we
have $m(E\backslash F(B))=0$ for each open ball $B$.

{\it Step 2.} Since $M$ is compact, there exists a countable
collection of open balls $\{B_n:n\ge1\}$ which forms a basis of the
topology on $M$. Let $F(B_n)$ be given by Step 1. We have
$m(E\backslash F)=0$ where $F=\bigcap_{n\ge1}F(B_n)$. Now for each
$x\in F$, $\mathcal{O}(x)\cap B_n\neq\emptyset$ for each $n\ge1$. So
the orbit $\mathcal{O}(x)$ is dense in $M$ for each point $x\in F$.
Equivalently we see $\mu$-a.e. $x\in M$ has a dense orbit. So the
$acip$ $\mu$ is weakly ergodic. This completes the proof.
\end{proof}
A natural question is, if $f\in\mathrm{SPH}^r(M)$ is essentially
accessible and preserves some $acip$ $\mu$, is $\mu$ an ergodic
measure? This is related to the uniqueness of $acip$. Clearly
uniqueness of $acip$ forces the ergodicity of $acip$. On the other
hand, let us assume that exists two $acip$'s: $\mu=\phi m$ and
$\nu=\psi m$. Let $E=\{x:\phi(x)>0\}$ and $F=\{x:\psi(x)>0\}$. If
$m(E\triangle F)>0$ we can further assume $E$ and $F$ are disjoint.
Proposition \ref{biess} implies that both $E$ and $F$ are
bi-essentially saturated (and nontrivial). In particular none of
them can be essentially bi-saturated.

We do not know whether such example can exist, or a bi-essentially
saturated set is automatically essentially bi-saturated. A
sufficient condition for this property is center bunching. From now
on we assume $r=2$ for simplicity.
\begin{definition}\label{centerb}
A strongly partially hyperbolic diffeomorphism $f$ is {\it center
bunched} if the functions $\nu$, $\tilde{\nu}$ and $\gamma$,
$\tilde{\gamma}$ given in \eqref{partial} can be chosen so that:
$\nu<\gamma\tilde{\gamma}$ and $\tilde{\nu}<\gamma\tilde{\gamma}$.
\end{definition}
\begin{proposition}[Corollary 5.2 in \cite{BW}]
Let $f\in \mathrm{SPH}^2(M)$ be center bunched. Then every
measurable bi-essentially saturated subset is essentially
bi-saturated.
\end{proposition}

\begin{corollary}\label{equi}
Let $f\in \mathrm{SPH}^2(M)$ be essentially accessible and center
bunched. If there exists some $acip$, then the $acip$ must be
equivalent to the volume.
\end{corollary}
\begin{proof}
Let $\mu$ be an $acip$ and $\phi$ be the density of $\mu$ with
respect to $m$. We showed that $E=\{x\in M:\phi(x)>0\}$ is
bi-essentially saturated. Center bunching implies that $E$ is also
essentially bi-saturated. Since $f$ is essentially accessible and
$m(E)>0$, $m(E)=1$ and hence $\mu$ is equivalent to the volume $m$.
\end{proof}

\begin{remark}
In \cite{BW}, a map $f$ is said to be {\it volume preserving} if $f$
preserves some invariant measure $\mu$ that is equivalent to the
volume. They proved that if $f\in \mathrm{SPH}^2(M)$ is essentially
accessible, center bunched and preserves some $\mu$ equivalent to
the volume, then the measure $\mu$ is ergodic (and $Kolmogorov$). It
is well known that ergodic measures either coincide or absolutely
singular with respect to each other. So by Corollary \ref{equi}, if
$f\in \mathrm{SPH}^2(M)$ is essentially accessible and center
bunched, then either $f$ is volume preserving in the board sense, or
there exists no $acip$ at all.
\end{remark}

Followed by Corollary \ref{equi} we get that the density
$\phi=\frac{d\mu}{dm}$ of an $acip$ is positive a.e. on $M$. Now we
use {\it Cohomologous Theory} developed in \cite{W} to show the
smoothness of the density of $\mu$. Namely let $\psi:M\to\mathbb{R}$
be a potential on $M$ and consider the cohomologous equation on $M$:
\begin{equation}\label{coho2}
\psi=\Psi\circ f-\Psi.
\end{equation}

\begin{proposition}[Theorem A, part $I\!I$ and $I\!I\!I$, in
\cite{W}]\label{smooth} Let $f\in \mathrm{SPH}^2(M)$ be accessible,
center bunched, and volume-preserving. Let $\psi:M\to\mathbb{R}$ be
a H\"older continuous potential. If there exists a measurable
solution $\Psi$ such that \eqref{coho2} holds for a.e. $x\in M$,
then there is a H\"older continuous solution $\Phi$ of \eqref{coho2}
with $\Phi=\Psi$ a.e. $x\in M$.
\end{proposition}

Let $\psi=-\log J_f$ and $\Psi=\log\phi$. Now $\psi$ is a $C^1$
function and $\Psi$ is a well defined measurable function. Corollary
\ref{equi} implies that $\Psi$ is a measurable solution of the
cohomologous equation \eqref{coho2}. Then applying Proposition
\ref{smooth} we get a H\"older continuous solution $\Phi$ of
\eqref{coho2} which coincides with $\Psi$ a.e.. It is evident that
$\mu=e^{\Phi}m$ and the density $e^{\Phi}$ is bounded and bounded
away from zero on $M$. Such a measure $\mu$ is called a {\it smooth
measure}. So we have

\begin{theorem}\label{dich}
Let $f\in \mathrm{SPH}^2(M)$ be accessible and center bunched. If
there exists some $acip$, then the $acip$ must have a H\"older
continuous density with respect to the volume of $M$ which is also
bounded and bounded away from $0$. In words, either $f$ preserves a
smooth measure or there is no $acip$ of $f$.
\end{theorem}

In particular center bunching holds whenever $E^c$ is
one-dimensional. As a corollary, we obtain:
\begin{corollary}
Let $f\in \mathrm{SPH}^2(M)$ be accessible and $\dim(E^c)=1$. Then
either $f$ preserves a smooth measure or there is no $acip$ of $f$.
\end{corollary}

Let $\mathrm{CB}^2(M)\subset\mathrm{SPH}^2(M)$ be the collection of
$C^2$ strongly partially hyperbolic diffeomorphisms that are center
bunched. Clearly $\mathrm{CB}^2(M)$ forms an open subset of
$\mathrm{SPH}^2(M)$. Applying Theorem \ref{dich} and the result in
\cite{DW} we have
\begin{theorem}\label{noacip}
The set of maps that admit no $acip$ contains a $C^1$ open and dense
subset of $\mathrm{CB}^2(M)$. In particular the set of maps that
admit no $acip$ contains a $C^1$ open and dense subset of $C^2$
strongly partially hyperbolic diffeomorphisms with $\dim(E^c)=1$.
\end{theorem}
The main obstruction for $C^2$ density in Theorem \ref{noacip} is
that we do not know whether stable accessibility is $C^2$ dense in
$\mathrm{SPH}^2(M)$.
\begin{proof}
Dolgopyat and Wilkinson proved in \cite{DW} that there is a $C^1$
dense subset of stable accessible diffeomorphisms in
$\mathrm{SPH}^2(M)$ (also $C^1$ dense in $\mathrm{CB}^2(M)$).
Starting with arbitrary $f\in\mathrm{CB}^2(M)$, we first perturb it
to a stable accessible one, say $f_1$. By $C^1$ closing lemma, there
exists $f_2\in\mathrm{CB}^2(M)$ close to $f_1$ that has some
periodic point. We can assume that $f_2$ is also stable accessible
since we can make it arbitrary close to $f_1$. By Franks' Lemma
\cite{Fr} we can assume that the periodic point $p$ is hyperbolic
with period $k$ and the Jacobian of $Df_2^k:T_xM \to T_xM$ has
absolute value different from $1$. These properties hold robustly
for all maps in a small neighborhood
$\mathcal{U}\subset\mathrm{CB}^2(M)$ of $f_2$.

Let $g\in\mathcal{U}$ and $p_g$ be the continuation of $p$. By the
choice of $\mathcal{U}$, we know that $g$ is accessible and center
bunched. If $g$ admits some $acip$ $\mu$, then by Theorem \ref{dich}
$\mu=\phi m$ for some H\"older continuous function $\phi$ which is
bounded and bounded away from zero. By Equation \eqref{coho1} we
have $\phi(p_g)=J_{g^k}(p_g)\phi(g^k p_g)=J_{g^k}(p_g)\phi(p_g)$.
This is impossible sice $|J_{g^k}(p_g)|\neq1$ and $\phi(p_g)\neq0$.
So each $g\in\mathcal{U}$ admits no $acip$. Hence there exists an
open set $\mathcal{U}$ close to $f$ in which each map admits no
$acip$. This finishes the proof.
\end{proof}

\begin{remark}
It is well known that among $C^2$ Anosov diffeomorphisms the ones
that admits no $acip$ are open and dense, see \cite[Corollary
4.15]{B1}. This is due to the fact that there are many periodic
points for every Anosov diffeomorphisms. Recently Avila and Bochi
\cite{AB} proved that a $C^1$-generic map in $C^1(M,M)$ has no
$acip$. In particular a $C^1$-generic map in $\mathrm{Diff}^1(M)$
has no $acip$.
\end{remark}

\section*{Acknowledgments} We would like to thank Wenxiang Sun,  Lan Wen,
Amie Wilkinson and Zhihong Xia for discussions and suggestions. We
are grateful to Shaobo Gan for useful comments and corrections to
the original manuscript. Especially we thank Amie Wilkinson for
explaining her results in \cite{W}.


\begin{thebibliography}{99}

\bibitem{AP} (MR2415085)
     \newblock J. Alves and V. Pinheiro,
     \newblock \emph{Topological structure of (partially) hyperbolic sets with positive volume},
     \newblock Trans. Amer. Math. Soc. \textbf{360} (2008), no. 10, 5551--5569.

\bibitem{AB} (MR2267725)
     \newblock A. Avila and J. Bochi,
     \newblock \emph{A generic $C^1$ map has no absolutely continuous invariant measure},
     \newblock Nonlinearity \textbf{19} (2006), 2717--2725.

\bibitem{BocV} (MR2090775)
     \newblock J. Bochi and M. Viana,
     \newblock \emph{Lyapunov exponents: how frequently are dynamical systems hyperbolic?},
     \newblock in Modern dynamical systems and applications, 271--297, Cambridge Univ. Press, Cambridge, 2004.

\bibitem{B} (MR0380890)
     \newblock R. Bowen,
     \newblock \emph{A horseshoe with positive measure},
     \newblock Invent. Math. \textbf{29} (1975), 203--204.

\bibitem{B1} (MR0442989)
     \newblock R. Bowen,
     \newblock \emph{Equilibrium states and the ergodic theory of Axiom A diffeomorphisms},
     \newblock Lecture Notes in Mathematics, Vol. \textbf{470}. Springer-Verlag, Berlin-New York, 1975.

\bibitem{Br}
     \newblock M. Brin,
     \newblock \emph{Topological transitivity of a certain class of dynamical systems, and flows of frames
on manifolds of negative curvature},
     \newblock Functional Anal. Appl. \textbf{9} (1975), 8--16.

\bibitem{BS} (MR1963683)
     \newblock M. Brin and G. Stuck,
     \newblock \emph{Introduction to dynamical systems},
     \newblock Cambridge University Press, 2002.

\bibitem{BW}
     \newblock K. Burns and A. Wilkinson,
     \newblock \emph{On the ergodicity of partially hyperbolic systems},
     \newblock Annals of Math. \textbf{171} (2010) 451--489.

\bibitem{BDP} (MR1933439)
     \newblock K. Burns, D. Dolgopyat and Ya. Pesin,
     \newblock \emph{Partial hyperbolicity, Lyapunov exponents and stable ergodicity},
     \newblock J. Statist. Phys. \textbf{108} (2002), no. 5-6, 927--942.

\bibitem{DW} (MR2039999)
     \newblock D. Dolgopyat and A Wilkinson,
     \newblock \emph{Stable accessibility is $C^1$ dense},
     \newblock  Geometric methods in dynamics. II.  Ast\'erisque  No. \textbf{287} (2003), xvii, 33--60.

\bibitem{F}
     \newblock T. Fisher,
     \newblock ``On the structure of hyperbolic sets",
     \newblock Ph.D thesis, Northwestern University, 2004.

\bibitem{Fr} (MR0283812)
     \newblock J. Franks,
     \newblock \emph{Necessary conditions for stability of diffeomorphisms},
     \newblock Trans. A.M.S. \textbf{158} (1971), 301--308.

\bibitem{G} (MR2371598)
     \newblock N. Gourmelon,
     \newblock \emph{Adapted metrics for dominated splittings},
     \newblock Ergod. Th. Dynam. Sys. \textbf{27} (2007), 1839--1849.

\bibitem{HPS} (MR0501173)
     \newblock M. Hirsch, C. Pugh and M. Shub,
     \newblock \emph{Invariant manifolds},
     \newblock Lecture Notes in Mathematics, Vol. \textbf{583}, Springer-Verlag, Berlin-New York, 1977.

\bibitem{NT} (MR1808220)
     \newblock V. Ni\c{t}ic\u{a} and A. T\"or\"ok,
     \newblock \emph{An open dense set of stably ergodic diffeomorphisms in a neighborhood of a non-ergodic one},
     \newblock Topology \textbf{40}  (2001),  no. 2, 259--278.

\bibitem{PS} (MR1750453)
     \newblock C. Pugh and M. Shub,
     \newblock \emph{Stable ergodicity and julienne quasiconformality},
     \newblock J. Eur. Math. Soc. \textbf{2} (2000), no. \textbf{1}, 1--52.

\bibitem{RY} (MR0590160)
     \newblock C. Robinson, L. S. Young,
     \newblock \emph{Nonabsolutely continuous foliations for an Anosov diffeomorphism},
     \newblock Invent. Math. \textbf{61} (1980), no. 2, 159--176.

\bibitem{RHRHU} (MR2388690)
     \newblock F. Rodriguez Hertz, M. Rodriguez Hertz and R. Ures,
     \newblock \emph{A survey of partially hyperbolic dynamics},
     \newblock Partially hyperbolic dynamics, laminations, and Teichm\"uller flow, Fields Inst. Commun.,
vol. \textbf{51}, Amer. Math. Soc., 2007, 35--87.

\bibitem{W}
     \newblock A. Wilkinson,
     \newblock \emph{The cohomological equation for partially hyperbolic diffeomorphisms}, arXiv:{0809.4862}.

\bibitem{X} (MR2220749)
     \newblock Z. Xia,
     \newblock \emph{Hyperbolic invariant sets with positive measures},
     \newblock Discrete Contin. Dyn. Syst. \textbf{15} (2006),  no. 3, 811--818.

\end{thebibliography}
\end{document}